\documentclass[10pt, reqno]{amsart}

\usepackage{amsmath, amssymb,amsthm}
\usepackage{hyperref}

\usepackage{latexsym}
\usepackage{fancyhdr}
\usepackage{pdflscape}
\usepackage{longtable}
\usepackage{rotating}
\usepackage{verbatim}
\usepackage{tikz}
\usepackage{subfigure}
\usepackage{mathrsfs}
\makeatletter
\newcommand{\mylabel}[2]{#2\def\@currentlabel{#2}\label{#1}}
\makeatother

\usepackage{mdwlist}
\usepackage{color}
\usepackage{tensor}

\usepackage[all]{xy}
\newdir{ >}{{}*!/-10pt/@{>}}
\newdir^{ (}{{}*!/-5pt/@^{(}}
\newdir_{ (}{{}*!/-5pt/@_{(}}

\newcounter{mnotecount}[section]

\newcommand{\rmnote}[1]{}

\theoremstyle{plain}
\newtheorem*{theorem*}{Theorem}
\newtheorem{theorem}{Theorem}[section]
\newtheorem*{lemma*}{Lemma}
\newtheorem{lemma}[theorem]{Lemma}
\newtheorem*{proposition*}{Proposition}
\newtheorem{proposition}[theorem]{Proposition}
\newtheorem*{corollary*}{Corollary}
\newtheorem{corollary}[theorem]{Corollary}
\newtheorem*{claim*}{Claim}

\newtheorem*{conjecture*}{Conjecture}

\newtheorem*{question*}{Question}
\theoremstyle{definition}
\newtheorem{remark}[theorem]{Remark}
\newtheorem*{remark*}{Remark}
\newtheorem*{definition*}{Definition}

\newtheorem*{example*}{Example}

\def\A{\;\forall}
\def\E{\;\exists}
\def\p{\partial}
\def\o{\circ}

\let\on=\operatorname

\newcommand{\ol}{\overline}
\newcommand{\ul}{\underline}

\def\db{\ol \partial}

\def\al{\alpha}

\def\ga{\gamma}
\def\de{\delta}
\def\ep{\varepsilon}
\def\ve{\varepsilon}
\def\ze{\zeta}

\def\la{\lambda}

\def\rh{\rho}

\def\si{\sigma}

\def\ta{\tau}

\def\ph{\varphi}
\def\vh{\varphi}

\def\ps{\psi}
\def\om{\omega}

\def\Ga{\Gamma}

\def\La{\Lambda}

\def\Ph{\Phi}
\def\Ps{\Psi}
\def\Om{\Omega}

\def\C{\mathbb{C}}

\def\K{\mathbb{K}}
\def\N{\mathbb{N}}

\def\R{\mathbb{R}}
\def\Z{\mathbb{Z}}

\def\cB{\mathcal{B}}
\def\cC{\mathcal{C}}

\def\cE{\mathcal{E}}

\def\cH{\mathcal{H}}

\def\cK{\mathcal{K}}

\def\fM{\mathfrak{M}}
\def\fN{\mathfrak{N}}

\def\sD{\mathscr{D}}

\title[Nonlinear conditions for ultradifferentiability]{Nonlinear conditions for ultradifferentiability: a uniform approach}

\author[D.N.~Nenning, A.~Rainer, and G.~Schindl]{David Nicolas Nenning, Armin Rainer, and Gerhard Schindl}

\address{Fakult\"at f\"ur Mathematik, Universit\"at Wien, Oskar-Morgenstern-Platz~1, A-1090 Wien, Austria.}
\email{david.nicolas.nenning@univie.ac.at}
\email{armin.rainer@univie.ac.at}
\email{gerhard.schindl@univie.ac.at}

\begin{document}

	\begin{abstract}
		Recent work showed that a theorem of Joris (that a function $f$ is smooth if
		two coprime powers of $f$ are smooth) is valid in a wide variety of
		ultradifferentiable classes $\cC$. The core of the proof was essentially $1$-dimensional.
		In certain cases a multidimensional version resulted from subtle reduction
		arguments, but general validity, notably in the quasianalytic setting, remained open.
		In this paper we give a uniform proof which works in all cases and dimensions.
		It yields the result even on infinite dimensional Banach spaces and
		convenient vector spaces.
		We also consider more general nonlinear conditions, namely general analytic germs $\Ph$
		instead of the powers, and characterize when $\Ph \o f \in \cC$ implies $f \in \cC$.
	\end{abstract}

\thanks{AR was supported by FWF-Project P 32905-N, DNN and GS by FWF-Project P 33417-N}

\maketitle

\section{Introduction}

A function $f$ is smooth provided that two relatively prime powers or, equivalently, two consecutive powers of $f$ are smooth,
by a theorem of Joris \cite{Joris82}.
In the recent paper \cite{Thilliez:2020ac} Thilliez showed that this result carries over to
Denjoy--Carleman classes of Roumieu type and,
by refining Thilliez's method, we proved in \cite{Nenning:2021wd} that it is valid in a wide variety of ultradifferentiable classes.
See the introduction in \cite{Nenning:2021wd} for more on the historical development.

In this analysis the dimension of the domain of $f$ had some substantial significance.
The proof in dimension one was based on holomorphic approximation
and utilized tools of complex analysis in one variable.
Multidimensional versions could be obtained only by some subtle reduction arguments (i.e., variants of Boman's theorem \cite{Boman67} and
Beurling-to-Roumieu reduction).
Consequently, general validity of the result in all dimensions remained open in cases, most notably the quasianalytic case,
where the suitable reduction tools are not available.
See the table in the introduction of \cite{Nenning:2021wd} for a summary of the known results
(in that paper the result in question was called  property ($\sD$)).

In the present paper we prove general validity in all dimensions and all cases.
In fact, the uniformity of our proof allows us to extend the result even to infinite dimensional Banach spaces and convenient vector spaces.
The crucial observation is that, in arbitrary dimension, holomorphic approximation in dimension one
can be replaced by uniform unidirectional holomorphic approximation combined with the polarization inequality.

Having proved that $\Ph \o f \in \cC$ implies $f \in \cC$ for $\Ph(t) = (t^j,t^{j+1})$, $j \in \N_{\ge 1}$, and $\cC$ a suitable ultradifferentiable class, it
is natural to ask what other $\Ph$ have this property. This question was studied by \cite{DuncanKrantzParks85} and \cite{JorisPreissmann87} in the $C^\infty$-setting.
By combining slightly adjusted arguments of \cite{JorisPreissmann87} with our result, we obtain a full characterization of the analytic germs
$\Ph : (\K,0) \to (\K^n,0)$ (where $\K$ is $\R$ or $\C$) having this property in terms of the support of the Taylor series of $\Ph$.
For example, it shows that $f \in \cC$ provided that $f^2 + f^3 \in \cC$ and $f^p \in \cC$, where $p$ is any positive integer.

The outline of the paper is the following.
Our main result is Theorem \ref{thm:main1}.
For its formulation we first need to define ultradifferentiable classes (Section \ref{sec:classes}) and
specify the required assumptions (Section \ref{sec:assumptions}).
The proof of Theorem \ref{thm:main1} is given in Section \ref{sec:proof}, and
in Section \ref{sec:Banach} we show that it extends to Banach spaces and even to convenient vector spaces.
Finally, in Section \ref{sec:nonlinear} we prove the characterization alluded to above.

\subsection{Ultradifferentiable classes} \label{sec:classes}

Let $\fM$ be a family of positive sequences $M=(M_k)_{k \ge 0}$.
Let $U \subseteq \R^d$ be an open subset.
The class $\cE^{\{\fM\}}(U)$ consists, by definition, of all complex valued functions $f \in C^\infty(U)$ with the following property:
for all compact subsets $K \subseteq U$ there exist $M \in \fM$ and $\rh>0$ such that
\begin{equation} \label{eq:def}
	\|f\|^M_{K,\rh} := \sup_{x \in K} \sup_{k \in \N}\frac{\|f^{(k)}(x)\|_{L^k(\R^d,\C)}}{\rh^k M_k} < \infty.
\end{equation}
Moreover, $\cE^{(\fM)}(U)$ consists of all complex valued functions $f \in C^\infty(U)$ such that
\eqref{eq:def} holds for all compact $K \subseteq U$, all $M \in \fM$, and all $\rh>0$.
Here $f^{(k)}$ denotes the Fr\'echet derivative of order $k$ of $f$ and
$\|f^{(k)}(x)\|_{L^k(\R^d,\C)} := \sup\{|f^{(k)}(x)(v_1,\ldots,v_k)| : \|v_i\| \le 1\}$
is the operator norm.
In view of the polarization inequality (cf.\ \cite[7.13.1]{KM97})
\begin{equation} \label{eq:polarization}
	\sup_{\|v\| \le 1} |d_v^k f(x)| \le \|f^{(k)}(x)\|_{L^k(\R^d,\C)} \le (2e)^k \sup_{\|v\| \le 1} |d_v^k f(x)|,
\end{equation}
where $d_v^k f(x) := \p_t^k f(x+tv)|_{t=0}$, we can equivalently use
\begin{equation} \label{eq:defv}
	\sup_{x \in K} \sup_{k \in \N} \sup_{\|v\|\le 1}\frac{|d_v^k f(x)|}{\rh^k M_k} < \infty
\end{equation}
instead of \eqref{eq:def} in the definition of
$\cE^{\{\fM\}}(U)$ and $\cE^{(\fM)}(U)$.

The global classes $\cB^{\{\fM\}}(U)$ and $\cB^{(\fM)}(U)$ are defined by taking the supremum over all $x \in U$
in \eqref{eq:def} (that is $\|f\|^M_{U,\rh} < \infty$) or equivalently \eqref{eq:defv}.
We shall also need the Banach space $\cB^M_\rh(U) := \{f \in C^\infty(U) : \|f\|^M_{U,\rh} < \infty\}$.

The classes $\cE^{\{\fM\}}$ and $\cB^{\{\fM\}}$ are said to be of \emph{Roumieu type},
$\cE^{(\fM)}$ and $\cB^{(\fM)}$ of \emph{Beurling type}.
By convention, we use $[\cdot]$ as a placeholder for $\{\cdot\}$ and $(\cdot)$.

\subsection{Admissible weight matrices} \label{sec:assumptions}

We need to impose some conditions on $\fM$.

\begin{definition*}
	A family $\fM$ of positive sequences $M=(M_k)_{k \ge 0}$ is said to be a \emph{weight matrix} if
	$\fM$ is totally ordered with respect to the pointwise order relation on sequences, and
	each $M \in \fM$ is log-convex with $M_0=1\le M_1$ and $M_k^{1/k} \to \infty$.
	A positive sequence $M$ with these properties is called a \emph{weight sequence}.
\end{definition*}

With $M$ we associate the sequence $m$ given by $k!\, m_k = M_k$.
Assuming that $m_k^{1/k} \to \infty$ we consider the function
\begin{equation*} \label{h}
	 h_{m}(t) := \inf_{k \in \N} m_k t^k, \quad \text{ for } t > 0, \quad \text{ and } \quad h_{m}(0):=0,
\end{equation*}
which is increasing, continuous on $[0,\infty)$, positive for $t>0$ and equals $1$ for large $t$.
For $t>0$ we put
\begin{align*}
	\label{counting2}
	\ol \Ga_{m}(t) &:= \min\{k : h_{m}(t) =  m_k t^k\}
\end{align*}
and
\begin{align*}
\ul \Ga_{m} (t) &:=  \min\Big\{k : \frac{m_{k+1}}{m_k}  \ge \frac{1}{t} \Big\}.
\end{align*}
Trivially, $\ul \Ga_m \le \ol \Ga_m$.
Equality $\ol \Ga_{m} = \ul \Ga_{m}$ holds,
if $m$ is log-convex.

For positive sequences $M$, $N$ we set
\begin{equation*}
\label{eq:rmatrixmg}
\on{mg}(M,N) := \sup_{j,k \ge 0, \, j+k\ge 1}\left(\frac{M_{j+k}}{N_jN_k}\right)^{1/(j+k)} \in (0,\infty].
\end{equation*}
Note that $\on{mg}(M,M) < \infty$ is the condition (M.2) of Komatsu \cite{Komatsu73} often called \emph{moderate growth}.

\begin{definition*}
	A weight matrix $\fM$ satisfying $m_k^{1/k} \to \infty$ for all $ M \in \fM$ is
	called
	\begin{itemize}
		\item $\{\text{\it admissible}\}$ or \emph{R-admissible}  if
		\begin{itemize}
			\item
			$\forall  M \in \fM \E  N \in \fM \E C\ge 1
			\A t>0 : \ol \Ga_{ n}(Ct) \le \ul \Ga_{ m}(t)$,
			\item
			$\forall M \in\fM \E N \in \fM : \on{mg}(M,N)<\infty.$
		\end{itemize}
		\item $(\text{\it admissible})$ or \emph{B-admissible}  if
		\begin{itemize}
			\item
			$\forall  M \in \fM \E  N \in \fM \E C\ge 1
			\A t>0 : \ol \Ga_{ m}(Ct) \le \ul \Ga_{ n}(t)$,
			\item
			$\forall M \in\fM \E N \in \fM : \on{mg}(N,M)<\infty.$
		\end{itemize}
	\end{itemize}
\end{definition*}

By our convention, $[\text{admissible}]$ stands for $\{\text{admissible}\}$ (i.e.\ R-admissible) in the Roumieu case and
$(\text{admissible})$ (i.e.\ B-admissible) in the Beurling case.

\begin{remark*}
		In the terminology of \cite{Nenning:2021wd} a weight matrix is $[\text{admissible}]$
		if and only if it is $[\text{regular}]$ and has $[\text{moderate growth}]$.
\end{remark*}

These are the minimal requirements needed for the tools used in our proof; see also
Remark \ref{rem:sharp} in which the necessity of these assumptions is discussed.
Note that [admissibility] of the weight matrix $\fM$ neither implies nor obstructs
quasianalyticity of the corresponding class $\cE^{[\fM]}$. In fact, (non-)quasianalyticity
plays no role in our analysis.

\subsection{Main results}

\begin{theorem}
	\label{thm:main1}
	Let $\fM$ be an [admissible] weight matrix.
	Let $U \subseteq \R^d$ be an open subset.
	Then a function $f : U \to \C$ belongs to
	$\cE^{[\fM]}(U)$ if $f^j, f^{j+1} \in \cE^{[\fM]}(U)$ for some positive integer $j$.
\end{theorem}

If we apply the theorem to the weight matrix which consists but of a single weight sequence $M$
(with $m_k^{1/k} \to \infty$, $\ol \Ga_{m}(Ct) \le \ul \Ga_{m}(t)$ for some $C\ge 1$, and $\on{mg}(M,M)<\infty$),
then we immediately get a version for Denjoy--Carleman classes.

As another special case we obtain the following result for Braun--Meise--Taylor classes $\cE^{[\om]}$.
For their definition and why it is a special case of Theorem \ref{thm:main1} we refer the reader to \cite[Section 4.5]{Nenning:2021wd}
and the references therein.
By a \emph{weight function} we mean a continuous increasing function $\om:[0,\infty) \rightarrow [0,\infty)$
such that
\begin{itemize}
	\item $\om(2t) = O (\om(t))$ as $t \rightarrow \infty$,
	\item $\om(t) = o(t)$  as $t \rightarrow \infty$,
	\item $\log(t) = o(\om(t))$  as $t \rightarrow \infty$,
	\item $t \mapsto \om(e^t)=:\vh_\om(t)$ is convex on $[0,\infty)$.
\end{itemize}

\begin{corollary} \label{cor:main1}
		Let $\om$ be a concave weight function.
		Let $U \subseteq \R^d$ be an open subset.
		Then a function $f : U \to \C$ belongs to
		$\cE^{[\om]}(U)$ if $f^j, f^{j+1} \in \cE^{[\om]}(U)$ for some positive integer $j$.
\end{corollary}

Actually, for a given weight function $\om$, the conclusion of Corollary \ref{cor:main1} holds \emph{if and only if}
$\om$ is equivalent to a concave weight function,
as we shall see in Remark~\ref{rem:sharp}. Two weight functions $\om_i$, $i=1,2$, are said to be \emph{equivalent} if
$\om_1(t)=O(\om_2(t))$ and $\om_2(t)=O(\om_1(t))$ as $t \to \infty$
which holds if and only if
$\cE^{[\om_1]} = \cE^{[\om_2]}$.

\section{Proof} \label{sec:proof}

\subsection{Preliminaries}

\begin{lemma}[{\cite[Lemma 5.1]{Nenning:2021wd}}]
\label{lem:hmoderategrowth}
	Let $M,N$ be weight sequences satisfying $m_k^{1/k} \to \infty$, $n_k^{1/k} \to \infty$, and $C:=\on{mg}(M,N)<\infty$. Then
	\begin{align}
		\label{eq:hmoderategrowth}
		h_m(t) &\le C^jn_j t^j h_n(Ct), \quad t>0, ~j \in \N,
		\\
		\label{eq:mgsquare}
		h_m(t) &\le h_n\Big(\frac{eC}{2}t\Big)^2, \quad t>0.
	\end{align}
\end{lemma}

For $\ep>0$ let $\Om_\ep$ be the interior of the ellipse in $\C$
with vertices $\pm \cosh(\varepsilon)$ and co-vertices $\pm i\sinh(\varepsilon)$.
By $\cH(\Om_\ep)$ we denote the space of holomorphic functions on $\Om_\ep$.
And $\|g\|_{X} := \sup_{z \in X} |g(z)|$ denotes the supremum norm.

The following lemma is a simple application of Hadamard's three lines theorem.

\begin{lemma}[{\cite[Lemma 5.2]{Nenning:2021wd} and \cite[Lemma 3.2.4]{Thilliez:2020ac}}]
	\label{lem:324v}
	Let $M,N$ be two weight sequences satisfying $m_k^{1/k} \to \infty$, $n_k^{1/k} \to \infty$,
	and $C:=\on{mg}(M,N)<\infty$. Let $\ep>0$.
	Let $g \in \cH(\Om_\ep)\cap C^0(\ol \Om_\ep)$ and assume that there are constants $L,a_1,a_2>0$ such that
	\[
		\|g\|_{\Om_\ep} \le L,\quad  \|g\|_{[-1,1]} \le a_1h_m(a_2\varepsilon).
	\]
	Then with $a_3:=\max\{a_1,L\}$ and $a_4:=eCa_2$
	we have
	\[
	\|g\|_{\Om_{\ep/2}} \le a_3 h_n(a_4 \varepsilon).
	\]
\end{lemma}

\subsection{Uniform unidirectional holomorphic approximation}

Let $\mathbb B := \{x \in \R^d : \|x\|\le 1\}$ be the closed unit ball in $\R^d$ and $\mathbb S := \{x \in \R^d : \|x\| =1\}$ the unit sphere.
For any open subset $U \subseteq \R^d$ consider
$V_U := \bigcup_{x \in U} (x+\mathbb B)$ which is again open.
If $f : V_U \to \C$ is a smooth function, then the composites
\[
		f_{x,v}(t):= f(x+tv), \quad t \in [-1,1],
\]
are well-defined and smooth for all $x \in U$ and $v \in \mathbb S$.
For notational convenience let $\La_U$ be the collection of all line segments $\la(t) = x+tv$, $t \in [-1,1]$,
where $x \in U$ and $v \in \mathbb S$, and $f_\la(t) := f_{x,v}(t) = f(x+tv)$.

The following theorem shows that ultradifferentiable bounds can be encoded by specific
uniform unidirectional holomorphic approximation. It is closely related to
\cite[Theorem 5.3]{Nenning:2021wd} (and \cite[Proposition 3.3.2]{Thilliez:2020ac}),
but special attention has to be paid to the uniformity (in $\la$) of all the estimates.

	\begin{theorem}
		\label{prop:332v}
		Let $U \subseteq \R^d$ be open and $f : V_U \to \C$ a smooth function.

			\thetag{i} Let $M^{(i)}$,  $1\le i \le 3$, be weight sequences with $(m^{(i)}_k)^{1/k} \rightarrow \infty$ and
			\begin{gather}
				\E B_1\ge 1 \A t>0 : \ol \Ga_{ m^{(2)}} (B_1 t) \le \ul \Ga_{ m^{(1)}}(t), \label{eq:countcomp}
				\\
				\E B_2\ge 1 \A j \in \N : m^{(2)}_{j+1} \le B_2^{j+1}m^{(3)}_j. \label{ass2}
			\end{gather}
			If $f \in \cB^{M^{(1)}}_{B_0}(V_U)$, then
			there exist constants $K \ge 1$,  $c_1,c_2>0$ and $0 < \ep_0\le \frac{1}2$
			and functions $f_{\la,\ve} \in \cH(\Om_\ep)\cap C^0(\ol \Om_\ep)$
			such that for all $0<\ep\le 2\ep_0$ and all $\la \in \La_U$,
			\begin{equation}
				\label{eq:holapprox}
				 \|f_{\la,\ve}\|_{\Om_\ep}\le K, \quad \|f_{\la}-f_{\la,\ve}\|_{[-1,1]}\le c_1 h_{m^{(3)}}(c_2\ve).
			\end{equation}
			The constants $K,c_1,c_2,\ep_0$ are independent of $\la$ and $\ep$, in particular, $c_2 = C B_0B_1$, where $C$ is an absolute constant.

			\thetag{ii} Let $N^{(i)}$, $1\le i \le 3$, be weight sequences with
			$(n^{(i)}_k)^{1/k} \rightarrow \infty$ and $\on{mg}(N^{(i)},N^{(i+1)})=D^{(i)}<\infty$.
			Assume that there exist constants $K \ge 1$,  $c_1,c_2>0$ and $0 < \ep_0\le \frac{1}2$
			and functions $f_{\la,\ve} \in \cH(\Om_\ep)\cap C^0(\ol \Om_\ep)$
			such that for all $0<\ep\le 2\ep_0$ and all $\la \in \La_U$,
			\begin{equation}
				\label{eq:holapprox2}
				 \|f_{\la,\ve}\|_{\Om_\ep}\le K, \quad \|f_{\la}-f_{\la,\ve}\|_{[-1,1]}\le c_1 h_{n^{(1)}}(c_2\ve),
			\end{equation}
			where the constants $K,c_1,c_2,\ep_0$ are independent of $\la$ and $\ep$.
			Then $f|_U \in \cB^{N^{(3)}}_\si(U)$,
			where $\si := \frac{4e^2 D^{(1)} D^{(2)}c_2}{a}$ and $a>0$ is an absolute constant.
	\end{theorem}

	\begin{proof}
	(i)
	The assumption $f \in \cB^{M^{(1)}}_{B_0}(V_U)$ implies
	\begin{equation*} \label{eq:asslines}
		\sup_{\la \in \La_U} \|f_{\la}^{(k)}\|_{[-1,1]} \le C_0 B_0^{k} M^{(1)}_k, \quad k \in \N.
	\end{equation*}
	Then there are, by \cite[Proposition 3.12]{FurdosNenningRainer} (thanks to \eqref{eq:countcomp} and \eqref{ass2}),
	constants $c_1,c_2>0$ and a function $F_{\la} \in C^1_c(\C)$
	extending $f_{\la}$ such that
	\begin{equation}
	\label{eq:dbext}
	|\db F_{\la} (z)| \le c_1 h_{m^{(3)}}(c_2 d(z,[-1,1])), \quad z \in \C.
	\end{equation}
	Note that the constants $c_1=c_1(C_0, B_0,B_1,B_2)$ and $c_2=12B_0B_1$ are independent of $\la$.
	By multiplication with a suitable cut-off function, we may assume that the support of $F_{\la}$ is contained in the
	disk $D$ centered at $0$ with radius $2$.
	Thus,
	\[
		F_{\la}(z) = \frac{1}{2 \pi i}\int_D \frac{\db F_{\la}(\ze)}{\ze -z} \, d \ze \wedge d\ol \ze, \quad z \in \C,
	\]
	and hence, in view of \eqref{eq:dbext}, $\|F_{\la}\|_\C$ is uniformly bounded in $\la$.

	The function
	$w_{\la,\ve}:=  \db F_{\la} \,\mathbf{1}_{\Om_\ve}$ satisfies
	\[
		\|w_{\la,\ep}\|_{\C} \le c_1 h_{m^{(3)}}(Cc_2 \ep),
	\]
	where $C>0$ is an absolute constant such that $d(z,[-1,1]) \le C\ep$ for $z \in \Om_\ep$.
	Then the bounded continuous function $v_{\la,\ve} = \cK * w_{\la,\ep}$, where $\cK$ is the Cauchy kernel in $\C$,
	satisfies $\db v_{\la,\ve} = w_{\la,\ve}$ in the distributional sense in $\C$ and
	$\|v_{\la,\ve}\|_{\C} \le C\,\|w_{\la,\ve}\|_{\C}$, where $C>0$ again is an absolute constant.
	So $f_{\la,\ve} :=F_{\la}- v_{\la,\ve}$ is holomorphic on $\Om_\ve$, continuous on $\ol \Om_\ep$, and fulfills \eqref{eq:holapprox}.

	(ii)
	For $\la \in \La_U$ and $0<\ep \le \ep_0$
	consider $g_{\la,\ep} := f_{\la,\ve} - f_{\la,2\ve} \in \cH(\Om_{\varepsilon}) \cap C^0(\ol \Om_{\varepsilon})$
	which satisfy $\|g_{\la,\ep}\|_{\Om_\ep} \le 2 K$  and
	$\|g_{\la,\ve}\|_{[-1,1]} \le 2c_1 h_{n^{(1)}}(2c_2\ve)$,
	by \eqref{eq:holapprox2}.
	Then
	Lemma~\ref{lem:324v} implies, with $c_3:= 2\max\{c_1,K\}$,
	\[
		\|g_{\la,\ve}\|_{\Om_{\varepsilon/2}} \le c_3 \, h_{n^{(2)}}(2eD^{(1)}c_2 \ve),
	\]
	for all $\la \in \La_U$ and $0<\ep \le \ep_0$.
	There exists a (universal) constant $a>0$ such that the closed disk with
	radius $a\ve$ around any $t \in [-\frac{1}{2},\frac{1}{2}]$ is contained in $\Om_{\varepsilon/2}$.
	The Cauchy estimates and \eqref{eq:hmoderategrowth} yield
	\begin{align*}
		\| g_{\la,\ve}^{(j)}\|_{[-\frac{1}{2},\frac{1}{2}]} &\le \frac{c_3\,j!\,   h_{n^{(2)}}(2eD^{(1)}c_2\varepsilon)}{(a\ep )^{j}}
		\\
		&\le c_3 \Big(\frac{2 eD^{(1)} D^{(2)} c_2}{a}\Big)^j  N_j^{(3)}\,
		h_{n^{(3)}}(2eD^{(1)} D^{(2)}c_2\varepsilon),
	\end{align*}
	which means that
	\[
		\| g_{\la,\ep}\|^{N^{(3)}}_{[-\frac{1}{2},\frac{1}{2}],\rh} \le c_3\,  h_{n^{(3)}}(2eD^{(1)} D^{(2)}c_2\varepsilon)
	\]
	for all $\la \in \La_U$ and $0<\ep \le \ep_0$,
	where
	$\rh := \frac{2eD^{(1)} D^{(2)}c_2}{a}$.
	In an analogous way the first bound in \eqref{eq:holapprox2} gives
	\begin{align*}
		\| f_{\la,\ve_0}^{(j)}\|_{[-\frac{1}{2},\frac{1}{2}]} &\le \frac{K}{(a\ve_0)^{j}}  j!
		\le
		A K \, \rh^j N^{(3)}_j.
	\end{align*}
	Indeed, the assumption $(n^{(3)}_k)^{1/k} \rightarrow \infty$ entails that
	for all $\ta>0$ there exists $A$ such that $j! \le A\, \ta^j N^{(3)}_j$ for all $j$;
	so take $\ta:= 2eD^{(1)} D^{(2)}c_2 \ep_0$.
	It follows that
	\[
	g_{\la} := f_{\la,\ve_0} + \sum_{j = 1}^\infty g_{\la,\ve_0 2^{-j}}
	= f_{\la,\ve_0} + \sum_{j = 1}^\infty (f_{\la,\ve_0 2^{-j}} - f_{\la,\ve_02^{-j+1}})
	\]
	converges absolutely in the Banach space $\cB^{N^{(3)}}_{\rh}([-\frac{1}{2},\frac{1}{2}])$ for each $\la \in \La_U$,
	and
	\begin{equation} \label{eq:goal}
		\sup_{\la \in \La_U} \|g_{\la}\|^{N^{(3)}}_{[-\frac{1}{2},\frac{1}{2}],\rh}  < \infty.
	\end{equation}
	We have $f_{\la} = g_{\la}$ on $[-\frac{1}{2},\frac{1}{2}]$, since
	$g_{\la} =  f_{\la,\ve_02^{-k}} + \sum_{j = k+1}^\infty (f_{\la,\ve_0 2^{-j}} - f_{\la,\ve_02^{-j+1}})$
	for every $k \in \N$, and so,
	for $t \in [-\frac{1}{2},\frac{1}{2}]$,
	\[
	|f_{\la}(t)-g_{\la}(t)| \le |f_{\la}(t)-f_{\la,\ve_02^{-k}}(t)|
	+ \Big|\sum_{j = k+1}^\infty (f_{\la,\ve_0 2^{-j}}(t) - f_{\la,\ve_02^{-j+1}}(t))\Big|
	\]
	which tends to $0$ as $k \to \infty$, by \eqref{eq:holapprox2} and absolute convergence of the sum.
	Now if $\la(t) = x+tv$ then
	$d^k_v f(x) = f^{(k)}_{\la}(0) = g^{(k)}_{\la}(0)$.
	Thus, \eqref{eq:goal} and the polarization inequality \eqref{eq:polarization},
	imply $f|_U \in \cB^{N^{(3)}}_\si(U)$ with $\si:= 2 e \rh$.
\end{proof}

\subsection{Proof of Theorem \ref{thm:main1}}
The next proposition extends \cite[Lemma 6.1]{Nenning:2021wd}; again uniformity in $\la$ is crucial.

\begin{proposition}
	\label{lem:corelemma}
	Let $U \subseteq \R^d$ be open and $f : V_U \to \C$ any function.
	Let $j$ be a positive integer.
	Let $M^{(i)}$, $1 \le i \le k$ with $k := \lceil\log_2(j(j+1))\rceil+7$,
	be weight sequences satisfying $(m^{(i)}_\ell)^{1/\ell} \rightarrow \infty$
	and
	\begin{gather*}
	\E B\ge 1 \A t>0 : \ol \Ga_{ m^{(2)}} (B t) \le \ul \Ga_{ m^{(1)}}(t),\\
	\on{mg}(M^{(i)},M^{(i+1)}) <\infty, \quad \text{ for }2 \le i \le k-1.
	\end{gather*}
	There is a constant $D>0$ such that $f^j,f^{j+1} \in \cB^{M^{(1)}}_{\rh}(V_U)$ implies $f \in \cB^{M^{(k)}}_{D\rh}(U)$.
	As a consequence
	$f^j,f^{j+1} \in \cB^{[M^{(1)}]}(V_U)$
	implies $f \in \cB^{[M^{(k)}]}(U)$.
\end{proposition}

\begin{proof}
	By the classical theorem of Joris, we may conclude that $f$ is smooth.

	Set $g:=f^j$ and $h := f^{j+1}$ and consider $g_{\la}$ and $h_{\la}$ for
	$\la \in \La_U$.
	(We may assume $j \ge 2$; otherwise the assertion is trivial.)
	By Theorem \ref{prop:332v}(i),
	there exist constants $K \ge 1$,  $c_1,c_2>0$ and $0 < \ep_0\le \frac{1}2$
	and functions $g_{\la,\ve}, h_{\la,\ve} \in \cH(\Om_\ep)\cap C^0(\ol \Om_\ep)$
	such that for all $0<\ep\le 2\ep_0$ and all $\la \in \La_U$,
	\begin{gather}\label{eq:ghbound}
		\max\{\|g_{\la,\ve}\|_{\Om_\ep},\|h_{\la,\ve}\|_{\Om_\ep}\} \le K,
		\\
		\label{eq:jj+1approx}
		\max\{\|g_{\la}- g_{\la,\ve}\|_{[-1,1]}, \|h_{\la}- h_{\la,\ve}\|_{[-1,1]}\} \le c_1h_{m^{(3)}}(c_2\ve).
	\end{gather}
	Using $|a^\ell - b^\ell| \le \ell \max\{|a|, |b|\}^{\ell-1}  |a - b|$, we find that
	$g_{\la,\ve}^{j+1} - h_{\la,\ve}^j \in \cH(\Om_\ep) \cap C^0(\ol \Om_\ep)$ satisfies
	\begin{align*}
		|g_{\la,\ve}^{j+1} - h_{\la,\ve}^j| &\le |g_{\la,\ve}^{j+1} - g_{\la}^{j+1}| + |h_{\la}^{j} - h_{\la,\ve}^j|
		\le c_3 h_{m^{(3)}}(c_2 \varepsilon), \quad \text{ on } [-1,1].
	\end{align*}
	The positive constant $c_3$ is again independent of $\la$ and $\ve$ as are all further constants $c_i$ appearing below.
	Thus Lemma \ref{lem:324v} implies
	\begin{equation}
	\label{eq:epshalf}
	\|h_{\la,\ve}^{j} - g_{\la,\ve}^{j+1}\|_{\Om_{\ep/2}} \le c_4 h_{m^{(4)}}(Ce c_2  \varepsilon) =: \de_\ep,
	\end{equation}
	where $C$ fulfills $C \ge \on{mg}(M^{(i)},M^{(i+1)})$ for all $2 \le i \le k-1$.
	Set $r_\ve:= \de_\ve^{\frac{1}{j+1}}$.
	We may assume that $\ep_0>0$ is chosen such that $\de_\ep \le r_\ep \le \frac{1}{2}$ for all $0<\ep\le 2\ep_0$.

	Consider the continuous function
	\[
	u_{\la,\ve} := \ph_{\ve}\frac{\ol g_{\la,\ve} h_{\la,\ve}}{\max\{|g_{\la,\ve}|, r_\ve\}^2},
	\]
	where $\ph_\ve$ is a smooth function compactly supported in $\Om_\ve$ and $1$ on $\Om_{\varepsilon/2}$.
	Note that $u_{\la,\ve}=h_{\la,\ep}/g_{\la,\ep}$ in $\Om_{\ep/2} \cap \{|g_{\la,\ep}|>r_\ep\}$.

	The uniformity of \eqref{eq:ghbound}, \eqref{eq:jj+1approx}, and \eqref{eq:epshalf}
	allows us to apply Lemmas 4.2.1 to 4.2.4 in \cite{Thilliez:2020ac} and conclude the following:
	for all $\la \in \La_U$ and $0< \ep \le 2\ep_0$,
	\begin{align} \label{eq:uepbound}
		\|u_{\la,\ep}\|_{\Om_{\ep/2}} &\le (2K)^{1/j},
		\\
		\label{eq:fuepbound}
		\|f_{\la}- u_{\la,\ep}\|_{[-1,1]} &\le c_5 r_\ve^{1/j},
	\end{align}
	where $c_5= c_5(K,j)$, cf.\ \cite[Lemma 4.2.2]{Thilliez:2020ac}.
	The bounded continuous function $v_{\la,\ep} = \cK * (\ol\p u_{\la,\ep}  \mathbf{1}_{\Om_{\ep/2}})$
	satisfies $\ol\p v_{\la,\ep} = \ol\p u_{\la,\ep}  \mathbf{1}_{\Om_{\ep/2}}$ in the distributional sense in $\C$ and
	\begin{equation} \label{eq:vepbound}
					\|v_{\la,\ep}\|_{\Om_{\ep/2}} \le c_6 \de_\ep^{1/s}
	\end{equation}
	where $s$ is any real number with $s > j(j+1)$ and $c_6= c_6(K,j,s)$, cf.\ \cite[Lemma 4.2.4]{Thilliez:2020ac} and
	also \cite[Lemma 3.1.1]{Thilliez:2020ac}.

	Then $f_{\la,\ve} := u_{\la,2\ve}-v_{\la,2\ve}$ is holomorphic in $\Om_\ep$ and continuous on $\C$.
	By \eqref{eq:uepbound}, \eqref{eq:fuepbound}, and \eqref{eq:vepbound},
	$\|f_{\la,\ep}\|_{\Om_{\ep}}$ is uniformly bounded and
	\[
	\|f_{\la}-f_{\la,\ve}\|_{[-1,1]} \le c_{7}\de_{2\ve}^{1/s}= c_7 \big(c_4 h_{m^{(4)}}(2e Cc_2  \varepsilon)\big)^{1/s},
	\]
	for all $\la \in \La_U$ and $0< \ep \le \ep_0$.
	Putting $s:=2^{k-6}=: 2^\ell$ and applying \eqref{eq:mgsquare} repeatedly,
	we find
	\begin{align*} \label{eq:final}
		\|f_{\la}-f_{\la,\ve}\|_{[-1,1]}
		\le
	 c_7 c_4^{1/s} h_{m^{(k-2)}}(2 c_2(eC)^{\ell+1}\ve).
 \end{align*}
	So, by Theorem \ref{prop:332v}(ii) (paying attention to the dependence of the constants), we may conclude that
	there is a uniform constant $D>0$ such that
	$f  \in \cB^{M^{(k)}}_{D \rh}(U)$
	as claimed.
\end{proof}

\begin{proof}[Proof of Theorem \ref{thm:main1}]
	It suffices to show that $f$ is locally of class $\cE^{[\fM]}$.
	Up to an affine transformation, we may assume that $g:= f^j$ and $h:= f^{j+1}$
	belong to $\cB^{[\fM]}(V_U)$
	and we will show that $f \in \cB^{[\fM]}(U)$, where $U$ is the open unit ball in $\R^d$.

	In the Roumieu case there exists $M^{(1)} \in \fM$ such that $g$, $h\in \cB^{\{M^{(1)}\}}(V_U)$.
	By R-admissibility of $\fM$,
	we find sequences $M^{(i)} \in \fM$ satisfying the assumptions of Proposition \ref{lem:corelemma}
	which implies that $f \in \cB^{\{M^{(k)}\}}(U)$.

	In the Beurling case we fix any $M \in \fM$ and show that $f \in \cB^{(M)}(U)$.
	Now B-admissibility of $\fM$ provides sequences $M^{(i)} \in \fM$
	as required in Proposition \ref{lem:corelemma} with $M^{(k)}=M$.
	As $g$, $h\in \cB^{(M^{(1)})}(V_U)$, Proposition \ref{lem:corelemma} gives  $f \in \cB^{(M)}(U)$.
\end{proof}

\begin{remark} \label{rem:bd1}
	Let $j$ be some positive integer. The map $\ph : \C^U \to \C^U \times \C^U$, $f \mapsto (f^j,f^{j+1})$ is injective.
	So there exists an inverse $\ps := \ph^{-1}|_{\ph(\C^U)}$.
	Under the assumptions of Theorem \ref{thm:main1}, the map $\ps$
	takes bounded sets in $\cE^{[\fM]}(U) \times \cE^{[\fM]}(U)$ to bounded sets in $\cE^{[\fM]}(U)$,
	if $\cE^{[\fM]}(U)$ is endowed with its natural locally convex topology;
	cf.\ \cite[Sec.\ 4.2]{RainerSchindl12}.
\end{remark}

\section{On Banach spaces and on convenient vector spaces} \label{sec:Banach}

\subsection{Ultradifferentiable classes of functions on Banach spaces}

Let $E$ be a Banach space and $U \subseteq E$ an open subset.
Let $\fM$ be a weight matrix.
For $M \in \fM$ and $\rh>0$ we consider $\cB^{M}_\rh(U) = \{f \in C^\infty(U,\C) : \|f\|^M_{U,\rh}<\infty\}$
as well as
\begin{align*}
		\cB^{\{\fM\}}(U) &:= \{f \in C^\infty(U,\C) : \E M \in \fM \E \rh >0 : \|f\|^M_{U,\rh}< \infty \} \quad \text{ and}
		\\
	 \cB^{(\fM)}(U) &:= \{f \in C^\infty(U,\C) : \A M \in \fM \A \rh >0 : \|f\|^M_{U,\rh}< \infty \},
\end{align*}
where $\|f\|^M_{U,\rh}$ is defined by \eqref{eq:def} (with $\R^d$ replaced by $E$).
Furthermore, we consider the local classes
\begin{align*}
	\cE^{\{\fM\}}(U) &:= \{f \in C^\infty(U,\C) : \A K \subseteq_{cp} U \E M \in \fM \E \rh >0 : \|f\|^M_{K,\rh}< \infty \}
	\quad \text{ and}
	\\
	 \cE^{(\fM)}(U) &:= \{f \in C^\infty(U,\C) : \A K \subseteq_{cp} U \A M \in \fM \A \rh >0 : \|f\|^M_{K,\rh}< \infty \}.
\end{align*}
Let $B(a,r) := \{x \in E : \|x-a\|<r\}$ be the open ball with center $a$ and radius $r$.

\begin{lemma} \label{lem:local}
	Let $f : U \to \C$ be smooth. Then:
	\begin{enumerate}
		\item $f \in \cE^{\{\fM\}}(U)$ if and only if
		\begin{equation} \label{eq:localR}
							\A a \in U \E M \in \fM  \E \rh >0 \E r>0 : f|_{B(a,r)} \in \cB^{M}_\rh(B(a,r)).
		\end{equation}
		\item $f \in \cE^{(\fM)}(U)$ if and only if
		\begin{equation} \label{eq:localB}
							\A a \in U \A M \in \fM  \A \rh >0 \E r>0 : f|_{B(a,r)} \in \cB^{M}_\rh(B(a,r)).
		\end{equation}
	\end{enumerate}
\end{lemma}

\begin{proof}
   (1)
	 That \eqref{eq:localR} is sufficient for $f \in \cE^{\{\fM\}}(U)$ is easily seen.
	 Fix a compact subset $K \subseteq U$.
	 For each $a \in K$ we get $M_a \in \fM$, $\rh_a>0$, and a small ball $B(a,r_a)$ such that
	 $f|_{B(a,r_a)} \in \cB^{M_a}_{\rh_a}(B(a,r_a))$.
	 Now $K$ is still covered by finitely many of the balls and we find $\|f\|^M_{K,\rh}< \infty$ with
	 $M:= \max M_a$ and $\rh:= \max \rh_a$, where the maxima are taken over the corresponding $a$.

	 Let us show that $f \in \cE^{\{\fM\}}(U)$ implies \eqref{eq:localR}.
	 If \eqref{eq:localR} does not hold, then
	 there exists $a \in U$ such that
	 \begin{align*}
		 &\A M \in \fM  \A \rh >0 \A r>0 \A C >0 \E x \in B(a,r) \E k \in \N :
		 \|f^{(k)}(x)\|_{L^k(E,\C)} \ge C \rh^k M_k.
	 \end{align*}
	 By \cite[Lemma 2.5]{FurdosNenningRainer}, we may assume that $\fM$ is indexed by $\N$ in the following way:
	 $\fM = \{M^{(\ell)} : \ell \in \N\}$ and $M^{(\ell_1)} \le M^{(\ell_2)}$ if and only if $\ell_1 \le \ell_2$.
	 In particular, taking $C=\rh=n$, $r = \tfrac{1}{n}$ and $M=M^{(n)}$, we find
	 \begin{align*}
		 &\A n \in \N
		 \E x_{n} \text{ with } \|x_{n} -a \|< \tfrac{1}{n}  \E k_{n} \in \N :
		 \| f^{(k_n)}(x_n)\|_{L^{k_n}(E,\C)} \ge n^{k_n+1} M^{(n)}_{k_n}.
	 \end{align*}
	 For the compact set $K := \{x_{n}\}_{n \in \N} \cup \{a\}$
	 this means that $\|f\|^M_{K,\rh}= \infty$ for all choices of $M \in \fM$ and $\rh>0$,
	 contradicting $f \in \cE^{\{\fM\}}(U)$.

	 (2)
	 That \eqref{eq:localB} is sufficient for $f \in \cE^{(\fM)}(U)$ is again easy to check.
	 For the necessity assume that \eqref{eq:localB} does not hold.
	 Then
	 there exist $a \in U$, $M \in \fM$, and $\rh>0$ such that
	 \begin{align*}
		 &\A r>0 \A C >0 \E x \in B(a,r)  \E k \in \N :
		 \|f^{(k)}(x)\|_{L^k(E,\C)} \ge C \rh^k M_k.
	 \end{align*}
	 Taking $C=n$ and $r = \tfrac{1}{n}$ we get
	 \begin{align*}
		 &\A n \in \N
		 \E x_{n} \text{ with } \|x_{n} -a \|< \tfrac{1}{n} \E k_{n} \in \N :
		 \|f^{(k_n)}(x_n)\|_{L^{k_n}(E,\C)} \ge n \, \rh^{k_n} M_{k_n}.
	 \end{align*}
	 Thus $\|f\|^M_{K,\rh}= \infty$ for $K := \{x_{n}\}_{n \in \N} \cup \{a\}$,
	 contradicting $f \in \cE^{(\fM)}(U)$.
\end{proof}

\subsection{Theorem \ref{thm:main1} holds on Banach spaces}

Let $E$ be a Banach space and $U \subseteq E$ an open subset.
Let $\mathbb B := \{x \in E : \|x\|\le 1\}$ be the closed unit ball in $E$ and $\mathbb S := \{x \in E : \|x\| =1\}$ the unit sphere.
Consider $V_U := \bigcup_{x \in U} (x + \mathbb B)$.

Then Theorem \ref{prop:332v} and Proposition \ref{lem:corelemma} remain valid for functions $f : V_U \to \C$ (with literally the same proofs).
Note that the classical Joris theorem is valid on Banach spaces by Boman's theorem.

\begin{theorem} \label{thm:Banach}
	Theorem \ref{thm:main1} (and thus Corollary \ref{cor:main1}) hold if $U$ is an open subset of a Banach space $E$.
\end{theorem}

\begin{proof}
	We treat Roumieu and Beurling case separately.

	\emph{Roumieu case.} Assume that $f^j,f^{j+1} \in \cE^{\{\fM\}}(U)$.
	Fix $a \in U$; without loss of generality we may assume that $a$ is the origin in $E$.
	By Lemma \ref{lem:local}, there exist $M \in \fM$, $\rh>0$, and $r>0$ such that $g:=f^j|_{B(0,r)}, h:=f^{j+1}|_{B(0,r)} \in \cB^{M}_\rh(B(0,r))$.
	Then $\tilde g(x) := g(\frac{r}{2}x)$ and $\tilde h(x) := h(\frac{r}{2}x)$ are elements of $\cB^{M}_{\frac{r}{2}\rh}(B(0,2))$;
	note that $B(0,2) = V_{B(0,1)}$.
	By R-admissibility of $\fM$,
	we find sequences $M^{(i)} \in \fM$ with $M^{(1)}=M$ satisfying the assumptions of Proposition \ref{lem:corelemma}
	which implies that the restriction of $\tilde f(x) := f(\frac{r}{2}x)$ to $B(0,1)$ belongs to $\cB^{M^{(k)}}_{D\frac{r}{2}\rh}(B(0,1))$.
	Then $f|_{B(0,\frac{r}{2})} \in \cB^{M^{(k)}}_{D\rh}(B(0,\frac{r}{2}))$ which ends the proof in view of Lemma \ref{lem:local}.

	\emph{Beurling case.} Assume that $f^j,f^{j+1} \in \cE^{(\fM)}(U)$.
	Fix $a \in U$, $M \in \fM$, and $\rh>0$; we may again assume that $a$ is the origin.
	By B-admissibility of $\fM$, there are sequences $M^{(i)} \in \fM$
	as required in Proposition \ref{lem:corelemma} with $M^{(k)}=M$.
	Let $D>0$ be the constant from Proposition \ref{lem:corelemma}.
	By Lemma \ref{lem:local}, there exists $r>0$ such that $g,h  \in \cB^{M^{(1)}}_{\frac{\rh}{D}}(B(0,r))$.
	Then $\tilde g,\tilde h  \in \cB^{M^{(1)}}_{\frac{r\rh}{2D}}(B(0,2))$
	(using the notation from the previous paragraph) and Proposition \ref{lem:corelemma} implies
	$\tilde f  \in \cB^{M^{(k)}}_{\frac{r\rh}{2}}(B(0,1))$
	and, consequently, $f|_{B(0,\frac{r}{2})}  \in \cB^{M}_{\rh}(B(0,\frac{r}2))$.
	Invoking Lemma \ref{lem:local} we are done.
\end{proof}

\subsection{Theorem \ref{thm:main1} holds on convenient vector spaces}

In \cite{KMRc,KMRq,KMRu} and \cite{Schindl14a} the calculus of $\cE^{[\fM]}$-maps has
been extended to maps between \emph{convenient vector spaces}, i.e., locally convex spaces
that are Mackey complete. We refer the reader to \cite{KM97} for a comprehensive treatment of convenient vector spaces
and the $c^\infty$-topology on them.

Let $E$ be a convenient vector space and $U \subseteq E$ a $c^\infty$-open subset. Let $\fM$ be a weight matrix.
Then $\cE^{[\fM]}(U)$ is by definition the set of all smooth functions $f : U \to \C$ (i.e.\ smooth on smooth curves in $U$)
such that for each closed absolutely convex bounded subset $B$ of $E$ we have
\[
	f \o i_B \in \cE^{[\fM]}(U_B),
\]
where $i_B : E_B \to E$ is the inclusion of the linear span $E_B$ of $B$ in $E$ and $U_B := U \cap E_B$.
Note that $E_B$ equipped with the Minkowski functional $p_B(x) = \inf \{t>0 : x \in t B\}$
is a Banach space, since $E$ is convenient.

\begin{theorem} \label{thm:convenient}
	Theorem \ref{thm:main1} (and thus Corollary \ref{cor:main1}) hold if $U$ is an $c^\infty$-open subset of a convenient vector space $E$.
\end{theorem}

\begin{proof}
	This is obvious by the definition of $\cE^{[\fM]}(U)$ and Theorem \ref{thm:Banach}.
\end{proof}

\section{More general nonlinear conditions} \label{sec:nonlinear}

As seen above the map
$\Ph(t) = (t^j,t^{j+1})$, where $j$ is some positive integer, has the property
\begin{equation} \label{eq:pseudoimmersive}
	\Ph \o f \in \cE^{[\fM]} \implies f \in \cE^{[\fM]},
\end{equation}
provided that $\fM$ is an [admissible] weight matrix.
Here $t$ is a complex or a real variable.
Equivalently, the map $\Ph(t) = (t^p,t^{q})$, where $p$ and $q$ are positive coprime integers,
has property \eqref{eq:pseudoimmersive}; indeed, all integers $j \ge pq$ have the form $j = ap+bq$ with $a,b \in \N$.

The question arises which maps $\Ph$ have the property \eqref{eq:pseudoimmersive} for all continuous $f$. From now on
we take continuity of $f$ as a basic assumption.
(The continuity of $f$ was not an issue in Theorem \ref{thm:main1}, since one of the powers is necessarily odd and hence
has a global continuous inverse.)
This problem was investigated in \cite{DuncanKrantzParks85} and in \cite{JorisPreissmann87}
in the $C^\infty$-setting. In the latter article the smooth germs $\Ph : (\R,0) \to (\R^n,0)$
(and $\Ph : (\C,0) \to (\C^n,0)$ which possess a complex Taylor series, cf.\ \cite[p.204]{JorisPreissmann87})
with the property that $\Ph \o f \in C^\infty$ implies $f \in C^\infty$ for all continuous germs $f$ with $f(0)=0$
were characterized in terms of a condition on the support of the formal Taylor series of $f$.
They were called \emph{pseudo-immersive} germs.

The proof is based (similar to most proofs in \cite{DuncanKrantzParks85}) on a reduction to the
case $\Ph(t) = (t^p,t^q)$ with $\on{gcd}(p,q)=1$ which is mostly algebraic and hence applies to our situation with marginal adjustments.
We thus obtain a characterization of the \emph{analytic} germs $\Ph$ satisfying \eqref{eq:pseudoimmersive}
for all continuous germs $f$ with $f(0)=0$ by the same condition on the support of the Taylor series.
Note that [admissibility] of the weight matrix $\fM$ implies that $\cE^{[\fM]}$ contains all analytic maps
and is stable under composition (by Corollaries 3.3 and 3.5 and Proposition 1.1 in \cite{FurdosNenningRainer}).
See also Remark~\ref{rem:ultra}.

\subsection{A necessary condition}

Let $\K$ be $\R$ or $\C$.
Let $\Ph  : (\K,0) \to (\K^n,0)$ be the germ of a non-zero analytic map.
Then $\Ph$ is represented by its convergent Taylor series
\begin{equation*}
	\Ph(t) = \sum_{1 \le k < N} a_k t^{n_k}, \quad \text{ where } a_k \in \K^n \setminus \{0\}\text{ for all }k,
\end{equation*}
$1 \le n_1 < n_2 < \cdots$,
and $1 \le N \le \infty$. That means that $\{n_1,n_2,\ldots\}$ is the \emph{support} of the power series $\Ph(t)$.

If \eqref{eq:pseudoimmersive} holds for all continuous germs $f$ with $f(0)=0$, then $\on{gcd}(n_1, n_2, \ldots ) =1$.
Indeed, if there are integers $p\ge 2$ and $\ell_k$ such that $n_k = p \ell_k$ for all $k$,
then the power series $\sum_{k} a_k x^{\ell_k}$ is convergent, defines an analytic germ $\Ps$ so that $\Ph(t) = \Ps(t^p)$,
and it is easy to find a continuous germ $f \not\in C^1$ with $f(0)=0$ such that $\Ph \o f$ is analytic;
cf.\ \cite[Theorem 2]{JorisPreissmann87}.

\subsection{A characterization}

We will see that the condition $\on{gcd}(n_1, n_2, \ldots ) =1$ is also sufficient for \eqref{eq:pseudoimmersive}
at least up to equivalence of analytic germs.
Two analytic germs $\Ph,\Ps : (\K,0) \to (\K^n,0)$ are said to be \emph{equivalent} if there exist
germs of analytic diffeomorphisms $u : (\K^n,0) \to (\K^n,0)$ and $v : (\K,0) \to (\K,0)$ such that
$u \o \Ph \o v = \Ps$.
Equivalent germs either both satisfy or do not satisfy \eqref{eq:pseudoimmersive}.
Any non-zero analytic germ $\Ph : (\K,0) \to (\K^n,0)$ is
equivalent to a germ whose first component is a positive power.
So it is no restriction to assume $\Ph_1(t) = t^p$.

\begin{theorem} \label{thm:main2}
	Let $\Ph = (\Ph_1,\ldots,\Ph_n) : (\K,0) \to (\K^n,0)$ be an analytic germ such that
	$\Ph_1(t) = t^p$
	and let $\{n_1,n_2,\ldots\}$ be the support of the power series $\Ph(t)$.
	Then the following conditions are equivalent:
	\begin{enumerate}
		\item
		Let $\fM$ be an [admissible] weight matrix.
		If $f : (\R^d,0) \to (\K,0)$ is a function germ and $\Ph \o f$ is of class $\cE^{[\fM]}$, then
		$f$ is of class $\cE^{[\fM]}$.
		\item
		Let $\om$ be a concave weight function.
		If $f : (\R^d,0) \to (\K,0)$ is a function germ and $\Ph \o f$ is of class $\cE^{[\om]}$, then
		$f$ is of class $\cE^{[\om]}$.
		\item $\on{gcd}(n_1, n_2, \ldots ) =1$.
	\end{enumerate}
\end{theorem}

\begin{proof}
	Since (2) is a special case of (1) (cf.\ \cite[Section 4.5]{Nenning:2021wd}) and the
	necessity of (3) for (2) follows from the discussion above,
	it remains to show that (3) implies (1).
	This is achieved by adjusting the arguments in the proofs of Theorems 3 and 3' in \cite{JorisPreissmann87};
 	actually the proof simplifies a bit, since all series involved in the arguments are convergent.

  We may assume that $p \ge 2$ and $n=2$.
	In fact, if $p=1$ there is nothing to prove and if $n>2$ then one
	can find real constants $\ga_i$ such that the support of the power series
	$(t^p,\ga_2 \Ph_2(t)+\cdots + \ga_n \Ph_n(t))$ satisfies (3).

	Thus we may assume that $\Ph(t) = (t^p, \vh(t))$, where $\ph(t) = \sum_{1 \le k < N} a_k t^{n_k}$
	with $a_k \in \K \setminus \{0\}$ for all $k$.
	Then (3) takes the form $\on{gcd}(p,n_1, n_2, \ldots ) =1$.
	The idea is to show that there exist a positive integer $q$ and analytic germs $\al_j$ at $0$ such that
	\begin{equation} \label{eq:key}
		t^{1+pq} = \sum_{j=0}^{p-1} \al_j(t^p) \vh(t)^j.
	\end{equation}
	Then (1) follows from Theorem \ref{thm:main1} (by plugging $f$ into \eqref{eq:key}).

	We will sketch how to establish \eqref{eq:key} following the steps in \cite{JorisPreissmann87};
	the main difference is that in \cite{JorisPreissmann87} all involved series are formal (power or Laurent) series
	so that at the end an application of Borel's lemma and some handling of the flat terms is required.

	Let $z$ be an indeterminate. Recall that $\K\{z\}$ denotes the ring of convergent power series in $z$ with coefficients in $\K$.
	The convergent power series
	$\vh(z) = \sum_{1 \le k < N} a_k z^{n_k}$ has the form
	\[
		\vh(z) = \vh_0(z^p) + z \vh_1(z^p) + \cdots + z^{p-1} \vh_{p-1}(z^p)
	\]
	where $\vh_j(z^p) \in \K\{z^p\}$ for all $j =0,\ldots,p-1$.
	Introducing another independent indeterminate $x$, set
	\[
		S(x) := \vh_0(z^p) + x \vh_1(z^p) + \cdots + x^{p-1} \vh_{p-1}(z^p) \in \K\{z^p\}[x].
	\]
	Furthermore, define
	\begin{align*}
		 G_{kj}(x) := S(x)^j x^{kp}, \quad j = 0,1, \ldots, p-1,~ k = 0,1,\ldots,(p-1)^2.
	\end{align*}

	Let $\K\{\!\{z\}\!\}$ denote the field of fractions of $\K\{z\}$. It is easy to see that
	$\K\{\!\{z\}\!\} = \{z^n g(z) : n \in \Z, \, g(z) \in \K\{z\}\}$.
	We look for nontrivial $a_{kj} \in \K\{\!\{z^p\}\!\}$ such that
	\[
		\sum_{j=0}^{p-1} \sum_{k=0}^{(p-1)^2} a_{kj} G_{kj}(x) = xH(x^p)
	\]
	with $H(x^p) \in \K\{\!\{z^p\}\!\}[x^p]$. This amounts to an underdetermined homogeneous system of linear equations
	over the field $\K\{\!\{z^p\}\!\}$; see \cite[p.201]{JorisPreissmann87} for details.
	So we find
	$a_{kj} \in \K\{\!\{z^p\}\!\}$ for $j=0,1,\ldots,p-1$, $k=0,1,\ldots,r$, and some $r\ge 0$, not all zero,
	and $H(x^p) \in \K\{\!\{z^p\}\!\}[x^p]$ such that
	\[
		\sum_{j=0}^{p-1} \sum_{k=0}^{r} a_{kj} G_{kj}(x) = xH(x^p).
	\]
	Regrouping the terms on the left-hand side, we obtain
	$A_j(x^p) \in \K\{\!\{z^p\}\!\}[x^p]$, $j=0,1,\ldots,p-1$, not all zero,
	such that
	\begin{equation} \label{eq:algebraic}
				\sum_{j=0}^{p-1}  A_{j}(x^p) S(x)^{j} = xH(x^p).
	\end{equation}
	Taking $r$ as small as possible, at least one of the series $A_0(z^p)$, ... , $A_{p-1}(z^p)$,
	which are obtained by replacing $x^p$ with $z^p$, is non-zero and also $H(z^p)$ is non-zero.
	Indeed, if $H(z^p)=0$ then \eqref{eq:algebraic} shows that $\vh(z)  = S(z)$ is algebraic over
	$\K\{\!\{z^p\}\!\} \subseteq \K(\!(z^p)\!) \subseteq \C(\!(z^p)\!)$
	of degree at most $p-1$; here $\K(\!(z)\!)$ is the field of formal Laurent series in $z$ with coefficients in $\K$.
	The automorphism $\si$ of $\C(\!(z)\!)$ induced by $z \mapsto \xi z$, where $\xi = e^{2\pi i/p}$, leaves invariant $\C(\!(z^p)\!)$ and
	generates the cyclic group $\langle\si\rangle$.
	Thus $\vh(z)$ (as an element of $\C(\!(z)\!)$) is invariant under a non-trivial subgroup of $\langle\si\rangle$,
	say $\langle\si^b\rangle$ with $1 \le b \le p-1$ and $b|p$.
	Then
	\[
		\sum_{1 \le k < N} a_k z^{n_k} = \sum_{1 \le k < N} a_k \xi^{b n_k}z^{n_k}
	\]
	which implies $\xi^{b n_k} =1$ for all $k$, since all $a_k \ne 0$.
	But that means that $p/b$ divides all $n_k$, contradicting $\on{gcd}(p,n_1, n_2, \ldots ) =1$.
	Cf.\ \cite[p.202]{JorisPreissmann87}.

	Multiplying \eqref{eq:algebraic} by a suitable power of $z^p$, we may suppose that the coefficients of $A_j$ and $H$
	belong to $z^p \K\{z^p\}$.
	Substituting $z$ for $x$, we find
	\[
		\sum_{j=0}^{p-1}  A_{j}(z^p) \vh(z)^j = zH(z^p)
	\]
	with $A_j(z^p)$ and $H(z^p)$ belonging to $z^p \K\{z^p\}$. Since $H(z^p) \ne 0$,
	there exist $q \in \N_{\ge 1}$,  $c \in \K \setminus \{0\}$, and $H_1(z^p) \in  \K\{z^p\}$ such that
	$H(z^p) = z^{pq}(c+ z^p H_1(z^p))$
	from which it is easy to conclude \eqref{eq:key}.
\end{proof}

\begin{remark}
	Let $\Ph$ be as in Theorem \ref{thm:main2} with $\on{gcd}(n_1, n_2, \ldots ) =1$
	and assume that $\fM$ is an [admissible] weight matrix.
	Then the map $\cE^{[\fM]} \ni \Ph \o f \mapsto f \in \cE^{[\fM]}$ in (1) takes bounded sets to bounded sets.
	This follows from Remark \ref{rem:bd1} and \eqref{eq:key}, since the superposition operator $g \mapsto h \o g$ is bounded,
	even continuous, among
	$\cE^{[\fM]}$ (in finite dimensions), provided that $\cE^{[\fM]}$ is stable under composition;
	see \cite[Theorem 4.13]{RainerSchindl12}.
\end{remark}

\begin{remark}
	\label{rem:sharp}
	Let $g$ be a continuous germ at $0 \in \R$ such that $g(0) =1$.
	Then $f := 1/g$ and $h := f - 1$ are continuous germs and $h(0) =0$.
	Consider the analytic germ $\vh(t) = \frac{1}{1+t} - 1 =  \sum_{k \ge 1} (-1)^k t^k$.
	Then $\vh \o h = g - 1$.
	Any analytic germ $\Ph : (\R,0) \to (\R^n,0)$ that has $\vh$ as a component evidently satisfies the condition \ref{thm:main2}(3).
	Now if $g$ is of class $\cE^{[\fM]}$ and \eqref{eq:pseudoimmersive} holds, then
	$h$ is of class $\cE^{[\fM]}$ and consequently also $f=1/g$.
	That means that \eqref{eq:pseudoimmersive} for analytic $\Ph$ implies that $\cE^{[\fM]}$ is \emph{inverse-closed}.

	In view of \cite[Theorem 4.8]{FurdosNenningRainer} we infer that concavity of the weight function $\om$ (up to equivalence)
	is a necessary condition
	for the conclusions of Corollary \ref{cor:main1} and Theorem \ref{thm:main2}(2).

	Next we want to discuss implications of the conclusions of Theorem \ref{thm:main1} and Theorem \ref{thm:main2}(1)
	and compare them with the assumption of admissibility for the weight matrix $\fM$;
	our goal is to convince the reader that admissibility is not only sufficient but also ``close to'' necessary.
	Suppose that
	\begin{equation} \label{eq:Com}
	 m_k^{1/k}\to \infty \quad \text{ for all } M \in \fM;
	\end{equation}
	this guarantees that $\cE^{[\fM]}$ contains the real analytic class.
	Consider also the conditions
	\begin{align} \label{eq:dc}
		\begin{split}
			\A M \in \fM \E N \in \fN &: \sup_{k} \Big(\frac{M_{k+1}}{N_k}\Big)^{1/k} < \infty \qquad \text{ in the Roumieu case,}
			\\
			\A M \in \fM \E N \in \fN &: \sup_{k} \Big(\frac{N_{k+1}}{M_k}\Big)^{1/k} < \infty \qquad \text{ in the Beurling case,}
		\end{split}
	\end{align}
	which characterize stability under derivation of $\cE^{[\fM]}$.

	\begin{description}
		\item[Fact i] Assuming \eqref{eq:dc}, the conclusions of Theorem \ref{thm:main1} and Theorem \ref{thm:main2}(1) imply that $\cE^{[\fM]}$ is stable under composition.
	\end{description}

	Indeed, we saw that \eqref{eq:pseudoimmersive} for analytic $\Ph$ entails that $\cE^{[\fM]}$ is inverse-closed which in turn implies (see \cite[Theorem 4.9 and 4.11]{RainerSchindl12}
	and \cite{RainerSchindl14})
	\begin{align} \label{eq:almostincreasing}
		\begin{split}
			\A M \in \fM \E N \in \fN &: \sup_{j\le k} \frac{m_j^{1/j}}{n_k^{1/k}} < \infty \qquad \text{ in the Roumieu case,}
			\\
			\A M \in \fM \E N \in \fN &: \sup_{j\le k} \frac{n_j^{1/j}}{m_k^{1/k}} < \infty \qquad \text{ in the Beurling case.}
		\end{split}
	\end{align}
	Now \eqref{eq:almostincreasing}, in conjunction with \eqref{eq:dc},
	yields that $\cE^{[\fM]}$ is stable under composition (\emph{loc.\!\! cit.}).
	Conversely, since $\cE^{[\fM]}$ contains the real analytic class (thanks to \eqref{eq:Com}),
	stability under composition entails inverse-closedness and hence \eqref{eq:almostincreasing}.
	Note that \eqref{eq:dc} follows from the moderate growth part in the definition
	of an [admissible] weight matrix.

	\begin{description}
		\item[Fact ii] Assuming \eqref{eq:Com} and \eqref{eq:dc},
		the other part in the definition of admissibility, i.e.\ the control of $\ol \Ga_n$ by $\ul \Ga_m$ (resp.\ $\ol \Ga_m$ by $\ul \Ga_n$),
		also implies stability under composition.
	\end{description}

	This follows from almost analytic extension (see \cite[Remark 2.7]{FurdosNenningRainer}).
	But the property ``control of $\ol \Ga_n$ by $\ul \Ga_m$'' seems not to be \emph{substantially} stronger than stability under composition.
	This is evidenced by Fact iii and Fact iv.

	\begin{description}
		\item[Fact iii] For any weight function $\om$,
		the class $\cE^{[\om]}$ being stable under composition implies the existence of a weight matrix
		$\fM$ with $\cE^{[\om]} = \cE^{[\fM]}$ such that
		$m$ is log-convex and thus
		$\ol \Ga_m = \ul \Ga_m$ for all $M \in \fM$.
	\end{description}

	See \cite[Theorem 4.8]{FurdosNenningRainer} and \cite[Proposition 3]{Rainer:2020aa});
	in this case \eqref{eq:Com} and moderate growth (hence \eqref{eq:dc}) are automatic.

	\begin{description}
		\item[Fact iv]
		If a general weight matrix $\fM$ fulfills \eqref{eq:Com}, \eqref{eq:almostincreasing}, and
		\begin{align} \label{eq:qr}
			\begin{split}
				\A M \in \fM \E N \in \fN  &: \sup_k \frac{m_k}{m_{k-1}} n_k^{-1/k} < \infty \qquad \text{ in the Roumieu case,}
				\\
				\A M \in \fM \E N \in \fN &: \sup_k \frac{n_k}{n_{k-1}} m_k^{-1/k} < \infty \qquad \text{ in the Beurling case,}
			\end{split}
		\end{align}
		then $\fM$ admits the desired control of $\ol \Ga_n$ by $\ul \Ga_m$.
	\end{description}

	Indeed, in the Roumieu case, for given $M \in \fM$ there exist $L,N \in \fM$ and constants $C_i>0$ such that for all $j \le k$,
	\[
			\frac{m_j}{m_{j-1}} \stackrel{\eqref{eq:qr}}{\le} C_1 \ell_j^{1/j} \stackrel{\eqref{eq:almostincreasing}}{\le} C_2  n_k^{1/k} \le C_3 \frac{n_k}{n_{k-1}};
	\]
	the last inequality it true for every weight sequence.
	It follows that $\ol \Ga_n(C_3 t) \le \ul \Ga_m(t)$ for all $t>0$ (see \cite[Lemma 3.3]{Rainer:2019ac}).
	The Beurling case
 is analogous.
\end{remark}

\begin{remark}
	\label{rem:ultra}
	In principle, one can apply the above method also to the more general case that $\Ph$ is the germ of a $\cE^{[\fM]}$-mapping,
	but the ramification $z^p$ causes a loss of regularity and, since the series are formal,
	an ultradifferentiable Borel lemma must be used
	which introduces flat terms whose handling again entails a loss of regularity.
\end{remark}


\def\cprime{$'$}
\providecommand{\bysame}{\leavevmode\hbox to3em{\hrulefill}\thinspace}
\providecommand{\MR}{\relax\ifhmode\unskip\space\fi MR }
\providecommand{\MRhref}[2]{%
  \href{http://www.ams.org/mathscinet-getitem?mr=#1}{#2}
}
\providecommand{\href}[2]{#2}

\end{document}